\documentclass[12pt]{article}
\usepackage{amscd, amsmath, amssymb, amsfonts, amsthm}
\usepackage{euscript}
\usepackage{graphicx}
\usepackage[square, comma, sort&compress, numbers]{natbib}
\usepackage[colorlinks,linkcolor=blue,citecolor=magenta,anchorcolor=red, bookmarksopen=false]{hyperref}

\def\pref#1{(\ref{#1})}

\theoremstyle{plain}
\newtheorem{prop}{Proposition}[section]
\newtheorem{thm}[prop]{Theorem}
\newtheorem{lem}[prop]{Lemma}
\newtheorem{cor}[prop]{Corollary}

\theoremstyle{definition}

\makeatletter
\def\p@figure{Fig. }
\def\p@enumii{}

\makeatother

\newcommand{\V}{\mathrm{V}}
\newcommand{\E}{\mathrm{E}}
\newcommand{\N}{\mathrm{N}}

\newcommand{\B}{\mathrm{B}}
\newcommand{\sqc}{$\mathcal{SQC}$}
\renewcommand{\b}[1]{\overline{#1}}

\newcommand{\se}{\subseteq}

\newcommand{\link}{\mathrm{link}}
\newcommand{\sm}{\setminus }

\newcommand{\ifof}{if and only if }

\newcommand{\floor}[1]{\left\lfloor{#1}\right\rfloor}
\newcommand{\tohi}{\emptyset}

\newcommand{\n}{\mathbb{N}}

\newcommand{\ca}{\mathcal{A}}

\begin{document}

\title{On Gorenstein Circulant Graphs and Gorenstein \sqc\ Graphs}
\author{Ashkan Nikseresht and Mohammad Reza Oboudi\\
\it\small Department of Mathematics, Shiraz University,\\
\it\small 71457-13565, Shiraz, Iran\\
\it\small E-mail: ashkan\_nikseresht@yahoo.com\\
\it\small E-mail: mr\_oboudi@yahoo.com }
\date{}
\maketitle

\begin{abstract}
We characterize some graphs with a Gorenstein edge ideal. In particular, we show that if $G$ is a circulant graph
with vertex degree at most four or a circulant graph of the form $C_n(1,\ldots, d)$ for some $d\leq n/2$, then $G$ is
Gorenstein if and only if $G\cong tK_2$, $G\cong t\b{C_n}$ or $G\cong tC_{13}(1,5)$ for some integers $t$ and $n\geq
4$. Also we prove that if $G$ is a \sqc\ graph, then $G$ is Gorenstein if and only if each component of $G$ is either
an edge or a 5-cycle.
\end{abstract}


                                        \section{Introduction}

Throughout this paper, $K$ is a field, $S=K[x_1,\ldots, x_n]$ and $G$ denotes a simple undirected graph with vertex set
$\V(G)=\{v_1,\ldots, v_n\}$ and edge set $\E(G)$. Recall that the \emph{edge ideal} $I(G)$ of $G$ is the ideal of $S$
generated by $\{x_ix_j|v_iv_j\in \E(G)\}$. Many researchers have studied how algebraic properties of $S/I(G)$ relates
to the combinatorial properties of $G$ (see \cite{stanley, hibi, tri-free} and references therein). One of the
algebraic properties that recently has gained attention is the property of being Gorenstein. We say that $G$ is a
Gorenstein (resp, Cohen-Macaulay or CM for short) graph over $K$, if $S/I(G)$ is a Gorenstein (resp. CM) ring. When $G$
is Gorenstein (resp. CM) over every field, we say that $G$ is Gorenstein (resp. CM). In \cite{large girth} and
\cite{tri-free} a characterization of planar Gorenstein graphs of girth $\geq 4$ and triangle-free Gorenstein graphs,
respectively, is presented. Also in \cite{planar goren} a condition on a planar graph equivalent to being Gorenstein is
stated.

In this paper, first we recall some needed concepts and preliminary results. Then in Section 3, we characterize some
graphs with a Gorenstein edge ideal. In particular, we show that if $G$ is a circulant graph with vertex degree at most
four or a circulant graph of the form $C_n(1,\ldots, d)$ for some $d\leq n/2$, then $G$ is Gorenstein if and only if
$G\cong tK_2$, $G\cong t\b{C_n}$ or $G\cong tC_{13}(1,5)$ for some integers $t$ and $n\geq 4$. Also we prove that if
$G$ is a \sqc\ graph, then $G$ is Gorenstein if and only if each component of $G$ is either an edge or a 5-cycle.

                                        \section{Preliminaries}

Recall that a \emph{simplicial complex} $\Delta$ on the vertex set $V=\{v_1,\ldots, v_n\}$ is a family of subsets of
$V$ (called \emph{faces}) with the property that $\{v_i\}\in \Delta$ for each $i\in [n]=\{1,\ldots,n\}$ and if $A\se
B\in \Delta$, then $A\in \Delta$. In the sequel, $\Delta$ always denotes a simplicial complex. Thus the family
$\Delta(G)$ of all cliques of a graph $G$ is a simplicial complex called the \emph{clique complex} of $G$. Also
$\Delta(\b G)$ is called the \emph{independence complex} of $G$, where $\b G$ denotes the complement of $G$. Note that
the elements of $\Delta(\b G)$ are independent sets of $G$. The ideal of $S$ generated by $\{\prod_{v_i\in F} x_i|F\se
V$ is a non-face of $\Delta\}$ is called the \emph{Stanley-Reisner ideal} of $\Delta$ and is denoted by $I_\Delta$ and
$S/I_\Delta$ is called the \emph{Stanley-Reisner algebra} of $\Delta$ over $K$. Therefore we have $I_{\Delta(\b G)}=
I(G)$. Many researchers have studied the relation between combinatorial properties of $\Delta$ and algebraic properties
of $S/I_\Delta$, see for example \cite{hibi,stanley, bal ver dec, our chordal, my vdec} and their references.

By the dimension of a face $F$ of $\Delta$, we mean $|F|-1$ and the dimension of $\Delta$ is defined as
$\max\{\dim(F)|F\in \Delta\}$. Denote by $\alpha(G)$ the \emph{independence number} of $G$, that is, the maximum size
of an independent set of $G$. Note that $\alpha(G)=\dim \Delta(\b G)+1$. A graph $G$ is called \emph{well-covered}, if
all maximal independent sets of $G$ have size $\alpha(G)$ and it is said to be a \emph{W$_2$ graph}, if $|\V(G)|\geq 2$
and every pair of disjoint independent sets of $G$ are contained in two disjoint maximum independent sets. In some
texts, W$_2$ graphs are called 1-well-covered graphs. In the following lemma, we have collected some known results on
W$_2$ graphs and their relation to Gorenstein graphs.
\begin{lem}\label{W2}
\begin{enumerate}
\item \label{Goren=> W2} (\cite[Lemma 3.1]{large girth} or \cite[Lemma 3.5]{tri-free})If $G$ is a graph without
    isolated vertices and $G$ is Gorenstein over some field $K$, then $G$ is a W$_2$ graph.

\item \label{tri-free Goren} (\cite[Proposition 3.7]{tri-free}) If $G$ is triangle-free (that is, no subgraph of $G$
    is a triangle) and without isolated vertices, then $G$ is Gorenstein \ifof $G$ is W$_2$.


\item \label{W2 deg} (\cite[Theorem 4]{W2}) The degree of every vertex of a connected non-complete W$_2$ graph is at
    least 2.
\end{enumerate}
\end{lem}

Recall that if $F\in \Delta$, then $\link_\Delta(F)=\{A\sm F| F\se A\in \Delta\}$. Suppose that $F\se \V(G)$. By
$\N[F]$ we mean $F\cup \{v\in \V(G)| uv\in \E(G)$ for some $u\in F\}$ and we set $G_F=G\sm \N[F]$. Thus if $F$ is
independent, then $\link_{\Delta(\b G)} F= \Delta(\b{G_F})$. Also the polynomial $I(G,x)=\sum_{i=0}^{\alpha(G)}
a_ix^i$, where $a_i$ is the number of independent sets of size $i$ in $G$, is called the \emph{independence polynomial}
of $G$. It should be mentioned that, if we set $f_i=a_{i+1}$ for $-1\leq i<\alpha(G)$, then the vector $(f_{-1},
f_0,\ldots, f_{\alpha(G)-1})$ is called the $f$-vector of $\Delta(\b G)$ in the literature of combinatorial commutative
algebra.

Theorem 5.1 of \cite{stanley} states conditions on $\Delta$ equivalent to the statement ``$\Delta$ is Goresnstein over
$K$''. Applying this theorem to the independence complex of $G$, we deduce the following characterization of Gorenstein
graphs. The details of the proof can be found in \cite{alpha=3}.
\begin{thm}\label{goren graphs}
Suppose that $G$ is a graph without isolated vertices. Then $G$ is Gorenstein over $K$ \ifof $G$ is CM over $K$,
$I(G,-1)= (-1)^{\alpha(G)}$ and $\b {G_F}$ is a cycle of length at least $4$ for each independent set $F$ of $G$ with
size $\alpha(G)-2$. In particular, if $\alpha(G)=2$, then $G$ is Gorenstein \ifof $G$ is the complement of a cycle.
\end{thm}
%
%

The above results can be applied to see that there are few Gorenstein graphs in many well-known classes of graphs. For
example, let $G=K_n$ (the complete graph on $n$ vertices) for $n\geq 2$. Then $I(G,x)=1+nx$ and $\alpha(G)=-1$. Thus by
\pref{goren graphs}, $G$ is Gorenstein \ifof $n=2$. Therefore, the only Gorenstein complete graphs are $K_2$ and $K_1$.
Also it follows from \cite[Corollary 9.1.14]{hibi} and \cite{very well} that every CM bipartite graph and every CM very
well-covered graph (that is, a well-covered graph $G$ with $|\V(G)|= 2\alpha(G)$) has a vertex of degree 1, hence by
\pref{W2}, we see that the only connected Gorenstein bipartite or very well-covered graphs are $K_2$ and $K_1$. Also if
$G$ is  an $n$-cycle for $n\geq 3$, then $G$ is W$_2$ only if $n=3$ or $n=5$ and it follows from \pref{goren graphs},
that the only Gorenstein cycle is $C_5$. In the next section, we search for Gorenstein graphs in two other classes of
graphs, namely, \sqc\ graphs and circulant graphs.

                    \section{Gorenstein circulant and \sqc\ graphs}
In this section, we recall the definition of circulant and \sqc\ graphs and characterize Gorenstein \sqc\ graphs and
Gorenstein circulant graphs with certain conditions. First, we consider circulant graphs. Suppose that $n\geq 3$ and
$S\se \{1,\ldots,\floor{n/2}\}$. Then the \emph{circulant graph} $C_n(S)$ is the graph with vertex set $[n]$ such that
${i,j}$ is an edge of $C_n(S)$ if and only if $\min\{|i-j|, n-|i-j|\} \in S$. In other words, assuming that $i>j$, the
vertices $i$ and $j$ are adjacent \ifof $i-j\in S\cup (n-S)$, where by $n-S$ we mean $\{n-s|s\in S\}$. In
\cite{circulant}, well-covered circulant graphs are studied and in \cite{CM circulant}, CM circulant graphs with degree
3 and CM circulant graphs of the form $C_n(1,\ldots, d)$ are characterized. Here we characterize Gorenstein circulant
graphs with degree $\leq 4$ and Gorenstein circulant graphs of the form $C_n(1,\ldots, d)$.

\begin{thm}\label{Cn(1..d)}
Assume that $n\geq 3$, $d\leq \floor{n/2}$ and $K$ is an arbitrary field and let $G=C_n(1,\ldots, d)$. Then $G$ is
Gorenstein (over $K$) \ifof $n=2d+3$.
\end{thm}
\begin{proof}
If $\alpha(G)=1$, then $G$ is a complete graph. But no complete graph on more than two vertices is Gorenstein. So
$\alpha(G)\geq 2$. Note that each vertex of $G$ is adjacent exactly to the (cyclically) next and  previous $d$
vertices. Thus the set $A_t=\{1,d+2,2d+3, \ldots, 1+t(d+1)\}$ is an independent set of $G$, if $1+t(d+1)\leq n-d$ and
is a maximal independent set if $t=t_0=\floor{(n-d-1)/(d+1)}$. Since $G$ is well-covered, $t_0=\alpha(G)$. Assume that
$\alpha(G)\geq 3$. Set $x=1+t_0(d+1)$ and $F=A_{t_0}\sm\{1,x\}$. Then by \pref{goren graphs} $\b{G_F}$ is a cycle, that
is each vertex of $G_F$ is adjacent to all except two other vertices of $G_F$. Note that $G_F$ is the induced subgraph
of $G$ on vertices $x,x+1, \ldots, n, 1$. Since the vertices of $G_F$ which are not adjacent to 1 are $x, x+1, \ldots,
n-d$, we conclude that $n-d=x+1$ (see \pref{fig-Cn}). But then $x+1$ is adjacent to all vertices of $G_F$, except 1
which is a contradiction. Consequently, $\alpha(G)=2$ and by \pref{goren graphs}, $\bar{G}$ is a cycle. Thus degree of
each vertex of $G$ which is $2d$ must equal $n-3$, that is, $n=2d+3$ as required. Conversely, if $n=2d+3$, then $G$ is
the complement of cycle and hence Gorenstein.
\end{proof}

\begin{figure}[!ht]
\begin{center}
\includegraphics{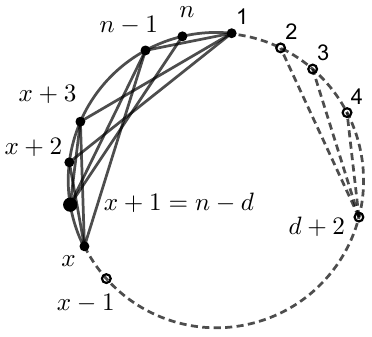}
\end{center}
\caption{The graphs $G$ and $G_F$ in the proof of \pref{Cn(1..d)}; the dashed part is not in $G_F$} \label{fig-Cn}
\end{figure}

Now we are going to characterize Gorenstein circulant graphs with vertex degree $\leq 4$. Note that  the only circulant
graphs with degree $1$ are $C_{2n}(n)$ which are disjoint unions of a set of edges and hence are Gorenstein. Also
connected circulant graphs with vertex degree $2$ are exactly the cycles $C_n$ which are Gorenstein \ifof $n=5$. In
\cite{CM circulant}, CM graphs with vertex degree 3 are characterized. Using their results we get the following. Note
that Gorenstein graphs with vertex degree $3$ are $C_{2n}(a,n)$ for some $1\leq a<n$.

\begin{thm}\label{cubic Goren circulant}
The circulant graph $G=C_{2n}(a,n)$ with $1\leq a<n$ is Gorenstein \ifof $2n/\gcd(a,2n)=3$. In this case, $G$ is
isomorphic to a disjoint copies of complement of a 6-cycle.
\end{thm}
\begin{proof}
Let $t=\gcd(a,2n)$ and assume that $G$ is Gorenstein. Since $G$ is CM and according to \cite[Theorem 5.5]{CM
circulant}, we get that $2n/t=3$ or $2n/t=4$. If $2n/t=4$, then by \cite[Theorem 5.3]{CM circulant}, $G$ is isomorphic
to $t$ copies of $C_4(1,2)=K_4$ which is not Gorenstein. So $2n/t=3$. Conversely, if $2n/t=3$, then by \cite[Theorem
5.3]{CM circulant}, $G$ is isomorphic to $t/2$ copies of $C_6(2,3)=\b{C_6}$ which is Gorenstein according to
\pref{goren graphs}.
\end{proof}


It remains to characterize circulant graphs with vertex degree equal to 4. These graphs are $C_n(a,b)$ with $1\leq
a<b<n/2$. For this, we need several lemmas. In the sequel, we always assume that $1\leq a<b<n/2$. It should be
mentioned that in what follows we have checked that some specific graphs are Gorenstein or W$_2$ or \ldots, using the
computer algebra system Macaulay 2 \cite{M2}.

\begin{lem}\label{W2lem}
Suppose that $B$ is an independent set of $G$ and $v\in B$. If $B\cap \N_G(x)\neq \{v\}$ for each $x\in \N_G(v)$, then
$G$ is not W$_2$.
\end{lem}
\begin{proof}
Suppose that $G$ is W$_2$. By replacing $B$ with a maximum independent set of $G$ containing $B$, we can assume that
$B$ is an independent set of size $\alpha(G)$. Let $A_1=\{v\}$ and $A_2=B\sm \{v\}$. Since $G$ is W$_2$, there are
disjoint maximum independent sets $B_1$ and $B_2$ containing $A_1$ and $A_2$, respectively. Assume that
$B_2=A_2\cup\{x\}$. As $x\notin B_1$, we have $x\neq v$. If $x$ is not adjacent to $v$, then since $B$ is a maximum
independent set of $G$, $x$ is adjacent to some vertex of $B\sm\{v\}=A_2\se B_2$, which contradicts independency of
$B_2$. Thus $x$ is adjacent to $v$. Hence by assumption, there is a vertex $y\in B\cap \N_G(x)$, with $y\neq v$. Again
$y\in A_2\se B_2$ and we get a contradiction. Consequently, $G$ is not W$_2$.
\end{proof}
\begin{lem}\label{gcd n,a,b=1}
Suppose that $G=C_n(a,b)$ and $d|n$ $d|a$ and $d|b$. Then $G$ is isomorphic to $d$ disjoint copies of
$C_{n/d}(a/d,b/d)$.
\end{lem}
\begin{proof}
For each $0\leq i<d$, let $V_i=\{v\in \V(G)| v\equiv i \mod{d}\}$. Note that for each such $i$, we have
$|(kd+i)-(k'd+i)| \in \{a,b\}$ \ifof $|(k-k')|\in \{a/d,b/d\}$. Thus the induced subgraph of $G$ on each $V_i$ is
isomorphic to $C_{n/d}(a/d,b/d)$ and it is clear that there is no edge in $G$ between $V_i$ and $V_j$ for $i\neq j$ and
the result follows.
\end{proof}

Note that when $\gcd(n,a)=1$, then $a$ has a multiplicative inverse modulo $n$, which we denote by $a^{-1}$.
\begin{lem}\label{gcd n,a=1}
If $\gcd(n,a)=1$ (resp. $\gcd(n,b)=1$), then $C_n(a,b)$ is isomorphic to $C_n(1,d)$ with $d=\min\{a^{-1}b \mod{n},
-a^{-1}b \mod{n}\}$ (resp. $d=\min\{b^{-1}a \mod{n}, -b^{-1}a \mod{n}\}$).
\end{lem}
\begin{proof}
Assume that $\gcd(n,a)=1$ (the proof for $\gcd(n,b)=1$ is similar). So $a$ has an inverse modulo $n$ and $d$ is
well-defined and the ``$\min$'' in the definition of $d$ guarantees $d< n/2$ . For each $i\in \V(C_n(1,d))$ set
$\phi(i)=1+(i-1)a \mod{n}$ and consider $\phi$ as a map $C_n(1,d)\to C_n(a,b)$. If for $1\leq i,i'\leq n$ we have
$\phi(i)=\phi(i')$, then $i\equiv i' \mod{n}$ (because $a$ is invertible modulo $n$) and hence $i=i'$. Thus $\phi$ is
one-to-one and onto. Also it is easy to verify that $\phi(i)$ and $\phi(i+1)$ and also $\phi(i)$ and $\phi(i+d)$are
adjacent in $C_n(a,b)$ (where we compute $i+1$ or $i+d$ modulo $n$, if $i+1>n$ or $i+d>n$, respectively). Since the two
graphs have the same number of edges, it follows that $\phi$ is an isomorphism.
\end{proof}

\begin{lem}\label{W2Cn1,2}
For $n\geq 4$, the graph $G=C_n(1,2)$ is W$_2$ \ifof $n\in \{4,5,7,8\}$.
\end{lem}
\begin{proof}
Suppose that $G$ is W$_2$ and hence well-covered. By \cite[Theorem 4.1]{circulant} we have $n\leq  8$ or $n=11$. One
can check that for $n=6$ or $11$, $G$ is not W$_2$ and for the other possible $n$'s, $G$ is W$_2$.
\end{proof}

Now we can characterize Gorenstein and W$_2$ circulant graphs of degree 4.
\begin{thm}\label{circulant 4 main}
Suppose that $1\leq a <b <n/2$. The graph $C_n(a,b)$ is Gorenstein (over $K$) \ifof $(n,a,b)\in \ca$ where
\begin{align}
 \ca= \{&(7d,d,2d), (7d,d,3d), (7d,2d,3d), (13d,d,5d), \notag\\
  & (13d,2d,3d),(13d,4d,6d)|d\in \n\}. \notag
\end{align}
Also $C_n(a,b)$ is W$_2$ \ifof
$$(n,a,b)\in \ca\cup\{(5d,d,2d), (8d, d,2d), (8d, 2d, 3d)|d\in \n\}.$$
\end{thm}
\begin{proof}
Assume that $G$ is W$_2$. By using \pref{gcd n,a,b=1}, we assume that $\gcd(n,a,b)=1$ (which corresponds to $d=1$).
First we prove the following claim.

\paragraph{\emph{Claim}.} One of the following relations hold: $n=2a+b$ or $n=a+2b$ or $n=2a+3b$ or $n=3a+2b$ or
$2n=2a+3b$ or $2n=3a+2b$.

\subparagraph{\it Proof of Claim.} Let $B=\{1,1+a+b,n-a- b+1\}$. Note that vertices $1+a$ and $1+b$ are adjacent to
$1+a+b$ and vertices $n-a+1$ and $n-b+1$ are adjacent to $n-a-b+1$, thus every neighbor of $1$ has a neighbor in
$B\sm\{1\}$. If $B$ is independent, then by \pref{W2lem} with $v=1$, we deduce that $G$ is not W$_2$. Therefore $B$ is
not independent. Suppose that $1+a+b\in \N(1)$. Then as $1+a+b>1$, we should have either $a+b\in \{a,b\}$ (which is not
possible) or $a+b\in \{n-a,n-b\}$. The latter condition is equivalent to either $n=2a+b$ or $n=a+2b$. Similarly if
$n-a-b+1\in \N(1)$, then again either $n=2a+b$ or $n=a+2b$.

Now assume that $1+a+b\in \N(n-a-b+1)$. If $1+a+b<n-a-b+1$, then $n-2a-2b\in \{a,b,n-a,n-b\}$ which implies that either
$n=2a+3b$ or $n=3a+2b$. Finally if $n-a-b+1<1+a+b$, then we must have $2a+2b-n\in \{a,b,n-a,n-b\}$, that is, either
$n=a+2b$ or $n=2a+b$ or $2n=3a+2b$ or $2n=2a+3b$. This concludes the proof of the claim. Now we consider several cases.

\paragraph{\emph{Case 1:} \it $n=2a+b$ or $n=a+2b$.} Assume that $n=a+2b$ and let $d=\gcd(n,b)$. Then $d|a$ and hence
$d=\gcd(n,a,b)=1$. Note that $a\equiv 2b \mod{n}$ and hence $b^{-1}a \equiv 2 \mod{n}$. Thus according to \pref{gcd
n,a=1}, $G$ is isomorphic to $C_n(1,2)$. Also if $n=2a+b$, a similar argument shows that $G\cong C_n(1,2)$. Note that
since degree of each vertex in $G$ is 4, $n\geq 5$. Therefore, it follows from \pref{W2Cn1,2}, that $n\in\{5,7,8\}$.
Consequently, in this case $G$ is W$_2$ \ifof $(n,a,b)$ equals one of the following: $(5,1,2)$, $(7,1,3)$, $(7,2,3)$ or
$(8,2,3)$. Moreover, according to \pref{Cn(1..d)}, in this case $G$ is Gorenstein \ifof $n=7$.

\paragraph{\emph{Case 2:} \it $n=2a+3b$.} Consider $B=\{1,1+2a, 1+b-a, n-a-b+1\}$. Every neighbor of 1 has a neighbor
in $B\sm\{1\}$. Hence by \pref{W2lem}, $B$ is not independent. If $1+2a \in \N(1)$, then $2a\in \{a,b,n-a,n-b\}$. The
only possible case is that $2a=b$ (for example, if $2a=n-a$, then by replacing $n$ with $2a+3b$, it follows that
$a=3b$, which contradicts $a<b$). But then $n=8a$ and hence $a=\gcd(n,a,b)=1$, that is, $(n,a,b)=(8,1,2)$. Note that
$C_8(1,2)$ is W$_2$ but not Gorenstein, according to \pref{Cn(1..d)} and \pref{W2Cn1,2}. If $1+b-a\in \N(1)$, then
again the only possible case is $b-a=a$ which leads to $(n,a,b)=(8,1,2)$ as above. Also note that if $n-a-b+1\in
\N(1)$, then by the proof of the claim above, $n=2a+b$ or $n=a+2b$ which contradicts the assumption of this case.

Now assume $1+2a$ and $1+b-a$ are adjacent. If $1+b-a> 1+2a$, then we have $b-3a\in\{a,b,n-a,n-b\}$. But the only
possible case is $b-3a=a$, or equivalently $b=4a$. Similar to the case $b=2a$, we see that $a=\gcd(n,a,b)$=1 and
$(n,a,b)=(14,1,4)$. But $C_{14}(1,4)$ is not W$_2$. If $1+b-a<1+2a$, then we have $3a-b\in \{a,b,n-a,n-b\}$. It follows
that either $b=2a$ or $2b=3a$. The former case was dealt above. In the latter case, we have $b=3t$ for some integer $t$
and $a=2t$, hence $n=13t$ and $t=\gcd(n,a,b)=1$, that is, $(n,a,b)=(13,2,3)$. Using Macaulay 2, we see that
$C_{13}(2,3)$ is W$_2$. Also no two neighbors of 1 in $C_{13}(2,3)$ are adjacent, which means, 1 is not in any triangle
of $C_{13}(2,3)$. But circulant graphs are vertex transitive and hence $C_{13}(2,3)$ is triangle-free and Gorenstein by
\pref{W2}\pref{tri-free Goren}.

Finally, noting that $\max\{1+2a, 1+b-a\}< n-b-a+1= a+2b+1$, we see that $\{1+2a,1+b-a\} \cap \N(n-b-a+1)\neq \tohi$
\ifof $2a+b$ or $2b-a$ are in $\{a,b,n-a,n-b\}$. But all of these cases lead to contradictions with $1\leq a<b<n/2$ and
are not possible. Consequently, under the assumption of case 2, $G$ is W$_2$ \ifof $(n,a,b)$ is $(8,1,2)$ or $(13,2,3)$
and only in the latter case $G$ is Gorenstein.

\paragraph{\emph{Case 3:} \it $n=3a+2b$.} Consider $B=\{1, 1+b-a, 1+b+a, n-b+a+1\}$. By \pref{W2lem}, $B$ is not
independent. Note that $v_1=1< v_2=1+b-a< v_3=1+b+a<v_4=n-b+a+1=1+4a+b$. Thus there are some $i$ and $j$ such that
$1\leq i<j\leq 4$ and  $v_j-v_i\in \{a,b,n-a,n-b\}$. Writing down all of these equations for $1\leq i<j\leq 4$ and
ruling out those which contradict $1\leq a<b<n/2$, we deduce that either $b=2a$ or $b=3a$ or $b=5a$ or $2b=3a$. Similar
to the above argument, in the first three cases, it follows that $a=\gcd(n,a,b)=1$ and in the last case
$a/2=b/3=\gcd(n,a,b)=1$. Therefore, $(n,a,b)$ is one of the following: $(7,1,2)$ or $(9,1,3)$ or $(13,1,5)$ or
$(12,2,3)$. Using Macaulay 2, we see that $G$ is W$_2$ only in the cases $(n,a,b)\in \{(7,1,2), (13,1,5)\}$. Also
$C_7(1,2)$ is Gorenstein by \pref{Cn(1..d)} and $C_{13}(1,5)$ is Gorenstein, because it is triangle-free and W$_2$.

\paragraph{\emph{Case 4:} \it $2n=2a+3b$.} First note that in this case $b$ is even and we have $b<2a$ (else
$n\leq 3b/2+b/2= 2b$, contradicting $b<n/2$). Let $B=\{1, 1+b-a, n-b-a+1=1+b/2, 1+2a\}$. By \pref{W2lem}, $B$ is not
independent and hence the difference of two of the elements of $B$ should be in $\{a,b,n-b,n-a\}$. The equations that
can be derived from this point and not contradicting $1\leq a<b<\min\{n/2,2a\}$ are: $3a=2b$, $4a=3b$ or $6a=5b$.
Assume that $3a=2b$, then $a=2t$ and $b=3t$. Since $b$ is even, $t=2t'$ and hence $n=a+3b/2=13t'$. Since
$t'=\gcd(n,a,b)=1$, we deduce that $(n,a,b)=(13,4,6)$. By a similar argument in other cases we have $(n,a,b)=(9,3,4)$
or $(14,5,6)$. Neither $C_{14}(5,6)$ nor $C_9(3,4)$ is W$_2$. Consequently, $G$ is W$_2$ \ifof $(n,a,b)=(13,4,6)$ and
if this is the case, then $G$ is also Gorenstein, because $C_{13}(4,6)$ is triangle-free.

\paragraph{\emph{Case 5:} \it $2n=3a+2b$.} In this case, $a$ is even and $b<3a/2$. Let $B=\{1< 1+b-a< 1+n-b-a= 1+a/2<
1+5a/2\}$. Note that $(1+5a/2)-(1+a)=n-b$ and hence $1+a$ is adjacent to $1+5a/2$. Similarly, each other neighbor of
$1$ is adjacent to an element of $B\sm\{1\}$. If $B$ is not independent, then the difference of two of its members
should be in $\{a,b,n-a,n-b\}$. But all of the equations that follow from this, contradict $1\leq
a<b<\min\{3a/2,n/2\}$. Hence $B$ is independent and by \pref{W2lem}, $G$ is not W$_2$, that is, this case is not
possible. This completes the proof of the theorem.
\end{proof}

We can simplify the above theorem, noting that several of the graphs in that theorem are isomorphic.
\begin{cor}
Let $G$ be a circulant graph with vertex degree $4$. Then $G$ is W$_2$ \ifof $G$ is isomorphic to disjoint copies of
one of the following graphs: $C_7(1,2), C_{13}(1,5), C_5(1,2), C_8(1,2)$. Moreover, $G$ is Gorenstein \ifof it is
isomorphic to a disjoint copies of $C_7(1,2)$ or to a disjoint copies of $C_{13}(1,5)$.
\end{cor}
\begin{proof}
Note that using \pref{gcd n,a=1}, we get $C_{13}(2,3)\cong C_{13}(4,6)\cong C_{13}(1,5)$, $C_7(1,2)\cong C_7(2,3)\cong
C_7(1,3)$ and $C_8(2,3)\cong C_8(1,2)$. Now the result follows from \pref{circulant 4 main} and \pref{gcd n,a,b=1}.
\end{proof}

Suppose that $H$ is a graph and $t\in \n$. By $tH$ we mean the graph consisting of $t$ disjoint copies of $H$.
Summarizing our results on Gorenstein circulant graphs, we get the following.
\begin{cor}
Suppose that $G$ is a circulant graph with vertex degree $\leq 4$ or a circulant graph of the form $C_n(1,\ldots, d)$
for some $d\leq n/2$. Then $G$ is Gorenstein \ifof $G\cong tK_2$, $G\cong t\b{C_n}$ or $G\cong tC_{13}(1,5)$ for some
$t\in \n$ and $4\leq n\in \n$.
\end{cor}


Next we consider \sqc\ graphs. Our interest in \sqc\ graphs arise form the fact that these graphs are CM and also every
CM graph with girth at least 5 or every CM graphs which does not contain any cycles of length 4 or 5 is in the class
\sqc, see \cite{large girth}. We recall the definition of \sqc\ graphs from \cite{block cactus}. By a \emph{simplicial
vertex} of $G$, we mean a vertex $v$ such that $\N_G[v]$ is complete and in this case we call $\N_G[v]$ a \emph{simplex
of $G$}. Furthermore, a 5-cycle $C$ of $G$ is called a\emph{ basic 5-cycle} when $C$ does not contain two adjacent
vertices both with degree more than 2. Also a 4-cycle $C$ is said to be\emph{ basic}, if it contains two adjacent
vertices of degree 2 and the two other vertices of $C$ each belong to a simplex or to a basic 5-cycle. For a basic
4-cycle $C$ we denote the set of degree 2 vertices of $C$ by $\B(C)$ and call them \emph{basic vertices} of $C$. If
there are simplices $S_1, \ldots, S_m$, basic 5-cycles $C_1, \ldots, C_t$ and basic 4-cycles $Q_1, \ldots, Q_r$ of $G$
such that
\begin{equation}\label{eq1}
\V(G)=\bigcup_{i=1}^m S_i \cup \bigcup_{i=1}^t \V(C_j) \cup \bigcup_{i=1}^r \B(Q_i)\qquad \hfill \tag{$*$}
\end{equation}
is a partition of $\V(G)$, we say that $G$ is in the class \sqc\ (or simply, $G$ is \sqc) and call \pref{eq1} a
\emph{\sqc\ partition} of $\V(G)$. Note that in this case, $G$ is well-covered and has independence number equal to
$m+2t+r$ (see \cite[Theorem 3.1]{block cactus}).

\begin{thm}\label{SQC}
Suppose that $G$ has no isolated vertex and is in the class \sqc. Then $G$ is Gorenstein \ifof each component of $G$ is
either an edge or a 5-cycle.
\end{thm}
\begin{proof}
Since $G$ is Gorenstein \ifof all of its components are so, we assume that $G$ is connected. Suppose that \pref{eq1} is
a \sqc\ partition of $\V(G)$. First note that $r=0$, since by \cite[Theorem 2]{pinter1} every vertex of a 4-cycle of
$G$ has degree at least 4 and $G$ cannot have any basic 4-cycles. Suppose that $S_i=\N[x_i]$. If $\deg(x_i)=1$ for some
$i$, then it follows from \pref{W2} that $G$ is just an edge. Thus assume that $\deg(x_i)\geq 2$ for each $i$.

Suppose that $m>1$. Note that in each basic 5-cycle of an arbitrary graph, there exist at least two non-adjacent
vertices of degree two. Let $F$ be the independent set consisting of $x_3,\ldots,x_m$ and two non-adjacent vertices of
degree two from each $C_i$. Then $|F|=\alpha(G)-2$ and by \pref{goren graphs} $\b{G_F}$ is a cycle. But as $x_1$ is not
adjacent to $N_G[x_2]$ in $G$ and $N_G[x_2]\se \V(G_F)$, we see that $\deg_{\b{G_F}}(x_1)\geq |\N_G[x_2]|\geq 3$, which
is a contradiction. If $m=1$, then $t>0$, else $G$ is a complete graph on at least 3 vertices and is not Gorenstein.
Let $F$ be the independent set consisting of one degree 2 vertex of $C_1$, say $y$, and two non-adjacent vertices of
degree 2 from each $C_j$ with $j>1$. Again $\b{G_F}$ should be a cycle but $deg_{\b{G_F}}(y)\geq |\N_G[x_1]|\geq 3$.
Therefore, we conclude that $m=0$.

Now suppose that $t\geq 2$. Because $G$ is connected, there is an edge between two vertices of two of the $C_i$'s, say
vertices $y_1$ and $y_2$ of $C_1$ and $C_2$, respectively. At least one of the vertices of $C_1$ which are not adjacent
to $y_1$ has degree 2, because $C_1$ is basic; call this vertex $y'_1$. Let $F$ to be the independent set containing
$y'_1,y_2$ and two non-adjacent vertices of degree 2 from each of the $C_i$ for $i>2$. Then $G_F$ is the disjoint union
of an edge and an isolated vertex which is not the complement of a cycle. This contradicts \pref{goren graphs} and we
conclude that $t=1$. Consequently, either $G= C_1$ is a 5-cycle, or $G$ is obtained from $C_1$ by adding one edge. In
the latter case, by applying \pref{goren graphs}, we see that $G$ is not Gorenstein and the result follows.
\end{proof}


\end{document}